\documentclass[12pt]{amsart}
\usepackage[margin=1in]{geometry}
\usepackage{amsmath}
\usepackage{amssymb}
\usepackage{amsthm}
\usepackage{enumerate}
\usepackage{yfonts}
\makeatletter
\def\imod#1{\allowbreak\mkern10mu({\operator@font mod}\,\,#1)}
\makeatother
\usepackage{graphicx}

\usepackage{cite}

\newtheorem{theorem}{Theorem}[section]
\newtheorem{lemma}{Lemma}[section]

\newtheorem{conjecture}{Conjecture}[section]

\theoremstyle{definition}

\newtheorem{remark}{Remark}[section]

\usepackage{mathrsfs}

\begin{document}

\title{Ranges of Unitary Divisor Functions}
\author{Colin Defant}
\thanks{This work was supported by National Science Foundation grant no. 1262930.}
\address{University of Florida}
\address{Current Address: Princeton University, Department of Mathematics}
\email{cdefant@ufl.edu, cdefant@princeton.edu}

\maketitle 

\begin{abstract}
For any real $t$, the unitary divisor function $\sigma_t^*$ is the multiplicative arithmetic function defined by $\sigma_t^*(p^{\alpha})=1+p^{\alpha t}$ for all primes $p$ and positive integers $\alpha$. Let $\overline{\sigma_t^*(\mathbb N)}$ denote the topological closure of the range of $\sigma_t^*$. We calculate an explicit constant $\eta^*\approx 1.9742550$ and show that $\overline{\sigma_{-r}^*(\mathbb N)}$ is connected if and only if $r\in(0,\eta^*]$. We end with some open problems.
\end{abstract}
\bigskip


  
\section{Introduction} 
For each $c\in\mathbb C$, the divisor function $\sigma_c$ is defined by $\sigma_c(n)=\sum_{d\mid n}d^c$. Divisor functions, especially $\sigma_1,\sigma_0$, and $\sigma_{-1}$, are among the most extensively-studied arithmetic functions \cite{Apostol, Hardy, Mitrinovic}. For example, two very classical number-theoretic topics are the study of perfect numbers and the study of friendly numbers. A positive integer $n$ is said to be \emph{perfect} if $\sigma_{-1}(n)=2$, and $n$ is said to be \emph{friendly} if there exists $m\neq n$ with $\sigma_{-1}(m)=\sigma_{-1}(n)$ \cite{Pollack}. Motivated by the very difficult problems related to perfect and friendly numbers, Laatsch \cite{Laatsch86} studied $\sigma_{-1}(\mathbb N)$, the range of $\sigma_{-1}$. He showed that $\sigma_{-1}(\mathbb N)$ is a dense subset of the interval $[1,\infty)$ and asked if $\sigma_{-1}(\mathbb N)$ is in fact equal to the set $\mathbb Q\cap[1,\infty)$. Weiner \cite{Weiner} answered this question in the negative, showing that $(\mathbb Q\cap[1,\infty))\setminus\sigma_{-1}(\mathbb N)$ is also dense in $[1,\infty)$.  

The author has studied ranges of divisor functions in a variety of contexts \cite{Defant,Defant5,Defant2,Defant3,Defant4}. For example, it is shown in \cite{Defant} that $\mathcal N(c)\to\infty$ as $\Re(c)\to-\infty$, where $\mathcal N(c)$ denotes the number of connected components of $\overline{\sigma_c(\mathbb N)}$. Here, the overline denotes the topological closure. In \cite{Sanna}, Sanna develops an algorithm that can be used to calculate $\overline{\sigma_{-r}(\mathbb N)}$ when $r>1$ is real and is known with sufficient precision. In addition, he proves that $\mathcal N(-r)$ is finite for such $r$. The author \cite{Defant5} has since extended this result, showing that $\mathcal N(c)$ is finite whenever $\Re(c)\leq 0$ and $c\neq 0$. Very recently, Zubrilina \cite{Nina} has obtained asymptotic estimates for $\mathcal N(-r)$ when $r>1$. She has also shown that there is no real number $r$ such that $\mathcal N(r)=4$. 

In this paper, we study the close relatives of the divisor functions known as unitary divisor functions. A \emph{unitary divisor} of an integer $n$ is a divisor $d$ of $n$ such that $\gcd(d,n/d)=1$. The unitary divisor function $\sigma_c^*$ is defined by  \cite{Alladi, Cohen, Guy} \[\sigma_c^*(n)=\sum_{\substack{d\mid n \\ \gcd(d,n/d)=1}}d^c.\] The function $\sigma_c^*$ is multiplicative and satisfies $\sigma_c^*(p^{\alpha})=1+p^{\alpha c}$ for all primes $p$ and positive integers $\alpha$. 

If $t\in[-1,0)$, then one may use the same argument that Laatsch employed in \cite{Laatsch86} in order to show that $\overline{\sigma_{t}^*(\mathbb N)}=[1,\infty)$. In particular, $\overline{\sigma_{t}^*(\mathbb N)}$ is connected if $t\in [-1,0)$. On the other hand, $\overline{\sigma_{t}^*(\mathbb N)}$ is a discrete disconnected set if $t\geq 0$ (indeed, in this case, $\sigma_t(\mathbb N)\cap[0,s]$ is finite for every $s>0$). The purpose of this paper is to prove the following theorem. Let $\zeta$ denote the Riemann zeta function. 

\begin{theorem} \label{Thm2.3} 
Let $\eta^*$ be the unique number in the interval $(1,2]$ that satisfies the equation 
\begin{equation}\label{Eq1}
\frac{2^{\eta^*}+1}{2^{\eta^*}}\cdot\frac{(3^{\eta^*}+1)^2}{3^{2\eta^*}+1}=\frac{\zeta(\eta^*)}{\zeta(2\eta^*)}.
\end{equation} 
If $r\in\mathbb R$, then $\overline{\sigma_{-r}(\mathbb N)}$ is connected if and only if $r\in(0,\eta^*]$. 
\end{theorem} 

\begin{remark}
In the process of proving Theorem~\ref{Thm2.3}, we will show that there is indeed a unique solution to the equation \eqref{Eq1} in the interval $(1,2]$. 
\end{remark}

In all that follows, we assume $r>1$ and study $\sigma_{-r}^*(\mathbb N)$. We first observe that $\sigma_{-r}^*(\mathbb N)\subseteq\displaystyle{\left[1,\zeta(r)/\zeta(2r)\right)}$. This is because if $q_1^{\beta_1}\cdots q_v^{\beta_v}$ is the prime factorization of some positive integer, then 
\[\sigma_{-r}^*(q_1^{\beta_1}\cdots q_v^{\beta_v})=\prod_{i=1}^v\sigma_{-r}^*(q_i^{\beta_i})=\prod_{i=1}^v\left(1+q_i^{-\beta_i r}\right)\leq\prod_{i=1}^v\left(1+q_i^{-r}\right)<\prod_{p}\left(1+p^{-r}\right)\] \[=\prod_{p}\left(\frac{1-p^{-2r}}{1-p^{-r}}\right)=\frac{\zeta(r)}{\zeta(2r)}.\]
It is straightforward to show that $1$ and $\zeta(r)$ are elements of $\overline{\sigma_{-r}^*(\mathbb N)}$. Therefore, Theorem~\ref{Thm2.3} tells us that $\overline{\sigma_{-r}^*(\mathbb N)}=\left[1,\zeta(r)/\zeta(2r)\right]$ if and only if $r\in(0,\eta^*]$. 

\section{Proofs} 

In what follows, let $p_i$ denote the $i^\text{th}$ prime number. Let $\nu_p(x)$ denote the exponent of the prime $p$ appearing in the prime factorization
of the integer $x$. 

To start, we need the following technical yet simple lemma. 

\begin{lemma} \label{Lem2.1} 
If $s,m\in\mathbb{N}$ and $s\leq m$, then $\displaystyle{\frac{p_s^{2r}+1}{p_s^{2r}+p_s^r}\leq\frac{p_m^{2r}+1}{p_m^{2r}+p_m^r}}$ for all $r>1$. 
\end{lemma} 
\begin{proof} 
Fix some $r>1$, and write $\displaystyle{h(x)=\frac{x^{2r}+1}{x^{2r}+x^r}}$. Then
\[h'(x)=\frac{r}{x(x^r+1)^2}\left(x^r-2-\frac{1}{x^r}\right).\] We see that $h(x)$ is increasing when $x\geq 3$. Hence, in order to complete the proof, it suffices to show that $h(2)\leq h(3)$. Let $f(s)=2^s3^{2s}+2^{2s}+2^s-(2^{2s}3^s+3^{2s}+3^s)$. 
For $s\geq 1$, we have  
\[f''(s)=18^s\log^2(18)+4^s\log^2(4)+2^s\log^2(2)-12^s\log^2(12)-9^s\log^2(9)-3^s\log^2(3)\] \[>18^s\log^2(18)-12^s\log^2(12)-9^s\log^2(9)>18^s\log^2(18)-2(12^s\log^2(12)).\] It is easy to verify that $18^s\log^2(18)-2(12^s\log^2(12))$ is increasing in $s$ for $s\geq 1$, so we obtain \[f''(s)>18\log^2(18)-2(12\log^2(12))>0.\] A simple calculation shows that $f'(1)>0$, so it follows that $f'(s)>0$ for all $s\geq 1$. Since $f(1)=0$ and $r>1$, we have $f(r)>0$. Equivalently, $2^{2r}3^r+3^{2r}+3^r<2^r3^{2r}+2^{2r}+2^r$. It follows that $(2^{2r}+1)(3^{2r}+3^r)<(2^{2r}+2^r)(3^{2r}+1)$. This shows that $\displaystyle{\frac{2^{2r}+1}{2^{2r}+2^r}<\frac{3^{2r}+1}{3^{2r}+3^r}}$, which completes the proof.  
\end{proof} 

The following theorem replaces the question of whether or not $\overline{\sigma_{-r}^*(\mathbb N)}$ is connected with a question concerning infinitely many inequalities. The advantage in doing this is that we will further reduce this problem to the consideration of a finite list of inequalities in Theorem~\ref{Thm2.2}. Recall from the introduction  that $\overline{\sigma_{-r}^*(\mathbb N)}$ is connected if and only if it is equal to the interval $[1,\zeta(r)/\zeta(2r)]$. 

\begin{theorem} \label{Thm2.1}
If $r>1$, then $\overline{\sigma_{-r}^*(\mathbb N)}=\displaystyle{\left[1,\zeta(r)/\zeta(2r)\right)}$ if and only if \[\frac{p_m^{2r}+p_m^r}{p_m^{2r}+1}\leq\prod_{i=m+1}^{\infty}\left(1+\frac{1}{p_i^r}\right)\] for all positive integers $m$. 
\end{theorem}
\begin{proof} 
First, suppose that $\displaystyle{\frac{p_m^{2r}+p_m^r}{p_m^{2r}+1}\leq\prod_{i=m+1}^{\infty}\left(1+\frac{1}{p_i^r}\right)}$ for all positive integers $m$. We will show that the range of $\log\sigma_{-r}^*$ is dense in $\displaystyle{\left[0,\log\left(\zeta(r)/\zeta(2r)\right)\right)}$, which will then imply that the range of $\sigma_{-r}^*$ is dense in $\displaystyle{\left[1,\zeta(r)/\zeta(2r)\right)}$. Fix some $\displaystyle{x\in\left(0,\log\left(\zeta(r)/\zeta(2r)\right)\right)}$. We will construct a sequence $(C_i)_{i=1}^{\infty}$ of elements of the range of $\log\sigma_{-r}^*$ that converges to $x$. First, let $C_0=0$. For each positive integer $n$, if $C_{n-1}<x$, let
$\displaystyle{C_n=C_{n-1}+\log\left(1+p_n^{-\alpha_n r}\right)}$, where $\alpha_n$ is the smallest positive integer that satisfies $\displaystyle{C_{n-1}+\log\left(1+p_n^{-\alpha_n r}\right)\leq x}$. If $C_{n-1}=x$, simply set $C_n=C_{n-1}=x$. For each $n\in\mathbb{N}$, $C_n\in\log\sigma_{-r}^*(\mathbb N)$. Indeed, if $C_n\neq C_{n-1}$, then  
\[C_n=\sum_{i=1}^n \log\left(1+p_i^{-\alpha_i r}\right)=\log\left(\prod_{i=1}^n\left(1+p_i^{-\alpha_i r}\right)\right)=\log\sigma_{-r}^*\left(\prod_{i=1}^np_i^{\alpha_i}\right).\] If, however, $C_n=C_{n-1}=x$, then we may let $l$ be the smallest positive integer such that $C_l=x$ and show, in the same manner as above, that $\displaystyle{C_n=C_l=\log\sigma_{-r}^*\left(\prod_{i=1}^lp_i^{\alpha_i}\right)}$. Let us write $\displaystyle{\gamma=\lim_{n\rightarrow\infty}C_n}$. Note that $\gamma$ exists and that $\gamma\leq x$ because the sequence $(C_i)_{i=1}^\infty$ is nondecreasing and bounded above by $x$. If we can show that $\gamma=x$, then we will be done. Therefore, let us assume instead that $\gamma<x$. 

We have $C_n=C_{n-1}+\log(1+p_n^{-\alpha_n r})$ for all positive integers $n$. Write
$D_n=\log(1+p_n^{-r})-\log(1+p_n^{-\alpha_n r})$ and $\displaystyle{E_n=\sum_{i=1}^n D_i}$. As 
\[x+\lim_{n\to\infty}E_n>\gamma+\lim_{n\to\infty}E_n=\lim_{n\rightarrow\infty}(C_n+E_n)=\lim_{n\rightarrow\infty}\left(\sum_{i=1}^n\log\left(1+p_i^{-\alpha_i r}\right)+\sum_{i=1}^n D_i\right)\] \[=\lim_{n\rightarrow\infty}\sum_{i=1}^n\log\left(1+p_i^{-r}\right)=\log\left(\zeta(r)/\zeta(2r)\right),\] we have $\displaystyle{\lim_{n\rightarrow\infty}E_n>\log\left(\zeta(r)/\zeta(2r)\right)-x}$. Therefore, we may let $m$ be the smallest positive integer such that $\displaystyle{E_m>\log\left(\zeta(r)/\zeta(2r)\right)-x}$. If $\alpha_m=1$ and $m>1$, then $D_m=0$. This forces $\displaystyle{E_{m-1}=E_m>\log\left(\zeta(r)/\zeta(2r)\right)-x}$, contradicting the minimality of $m$. If $\alpha_m=1$ and $m=1$, then $\displaystyle{0=E_m>\log\left(\zeta(r)/\zeta(2r)\right)-x}$, which is also a contradiction since we originally chose $x<\log(\zeta(r)/\zeta(2r))$. Therefore, $\alpha_m>1$. Due to the way we defined $C_m$ and $\alpha_m$, we have
$\displaystyle{C_{m-1}+\log\left(1+p_n^{-(\alpha_{m}-1)r}\right)>x}$. Hence,
\[\log\left(1+p_n^{-(\alpha_{m}-1)r}\right)-\log\left(1+p_n^{-\alpha_m r}\right)>x-C_m.\] Using our original assumption that 
$\displaystyle{\frac{p_m^{2r}+p_m^r}{p_m^{2r}+1}\leq\prod_{i=m+1}^{\infty}\left(1+\frac{1}{p_i^r}\right)}$, we have 
\[\log\left(\frac{p_m^{2r}+p_m^r}{p_m^{2r}+1}\right)\leq\sum_{i=m+1}^{\infty}\log\left(1+\frac{1}{p_i^r}\right)=\log\left(\frac{\zeta(r)}{\zeta(2r)}\right)-E_m-C_m\]
\[<x-C_m<\log\left(1+p_n^{-(\alpha_{m}-1)r}\right)-\log\left(1+p_n^{-\alpha_m r}\right)=\log\left(\frac{p_m^{\alpha_m r}+p_m^r}{p_m^{\alpha_m r}+1}\right).\]
Thus, 
\[\frac{p_m^{2r}+p_m^r}{p_m^{2r}+1}<\frac{p_m^{\alpha_m r}+p_m^r}{p_m^{\alpha_m r}+1}.\] 
Rewriting this inequality, we get $\displaystyle{p_m^{2r}+p_m^{(\alpha_m+1)r}<p_m^{3r}+p_m^{\alpha_m r}}$. 
Now, dividing through by $p_m^{\alpha_mr}$ yields $\displaystyle{p_m^{(2-\alpha_m)r}+p_m^r<1+p_m^{(3-\alpha_m)r}}$, which is impossible since 
$\alpha_m\geq 2$. This contradiction proves that $\gamma=x$, so $\overline{\sigma_{-r}^*(\mathbb N)}=\left[1,\zeta(r)/\zeta(2r)\right]$. 

To prove the converse, suppose there exists some positive integer $m$ such that \[\frac{p_m^{2r}+p_m^r}{p_m^{2r}+1}>\prod_{i=m+1}^{\infty}\left(1+\frac{1}{p_i^r}\right).\] We may write this inequality as 
\begin{equation} \label{EqEdit1}
\frac{p_m^{2r}+1}{p_m^{2r}+p_m^r}<\prod_{i=m+1}^{\infty}\left(1+\frac{1}{p_i^r}\right)^{-1}. 
\end{equation} 

Fix a positive integer $N$. If 
$\nu_{p_s}(N)=1$ for all $s\in\{1,2,\ldots,m\}$, then \[\sigma_{-r}^*(N)\geq\prod_{s=1}^m\left(1+\frac{1}{p_s^r}\right)=\frac{\zeta(r)}{\zeta(2r)}\prod_{i=m+1}^{\infty}\left(1+\frac{1}{p_i^r}\right)^{-1}.\] 
On the other hand, if $\nu_{p_s}(N)\neq 1$ for some $s\in\{1,2,\ldots,m\}$, then 
$\displaystyle{\sigma_{-r}^*\left(p_s^{\nu_{p_s}(N)}\right)\leq}$ $\displaystyle{1+\frac{1}{p_s^{2r}}}$. This implies that  
\[\sigma_{-r}^*(N)\leq\left(1+\frac{1}{p_s^{2r}}\right)\prod_{\substack{i=1 \\ i\neq s}}^{\infty}\left(1+\frac{1}{p_i^r}\right)=\frac{\zeta(r)}{\zeta(2r)}\frac{1+p_s^{-2r}}{1+p_s^{-r}}=\frac{\zeta(r)}{\zeta(2r)}\frac{p_s^{2r}+1}{p_s^{2r}+p_s^r}\] in this case.  
Using Lemma~\ref{Lem2.1}, we have 
\[\sigma_{-r}^*(N)\leq\frac{\zeta(r)}{\zeta(2r)}\frac{p_m^{2r}+1}{p_m^{2r}+p_m^r}.\] 
As $N$ was arbitrary, we have shown that there is no element of the range of $\sigma_{-r}^*$ in the interval \[\left(\frac{\zeta(r)}{\zeta(2r)}\frac{p_m^{2r}+1}{p_m^{2r}+p_m^r},\frac{\zeta(r)}{\zeta(2r)}\prod_{i=m+1}^{\infty}\left(1+\frac{1}{p_i^r}\right)^{-1}\right).\] This interval is a gap in the range of $\sigma_{-r}^*$ because of the inequality \eqref{EqEdit1}. 
\end{proof} 

As mentioned above, we wish to reduce the task of checking the infinite collection of inequalities given in Theorem~\ref{Thm2.1} to that of checking finitely many inequalities. We do so in Theorem~\ref{Thm2.2}, the proof of which requires the following lemma. 

\begin{lemma} \label{Lem2.2}
If $j\in\mathbb{N}\setminus\{1,2,3,4,6,9\}$, then $\displaystyle{\frac{p_{j+1}}{p_j}<\sqrt[3]{2}}$. 
\end{lemma}
\begin{proof}
In \cite{Nagura}, it is shown that $\dfrac{p_{j+1}}{p_j}\leq \dfrac{6}{5}<\sqrt[3]{2}$ for all $j\geq 10$. We easily verify the cases $j=5,7,8$ by hand.    
\end{proof} 

\begin{theorem} \label{Thm2.2} 
If $r\in(1,3]$, then $\overline{\sigma_{-r}^*(\mathbb N)}=\left[1,\zeta(r)/\zeta(2r)\right]$ if and only if \[\frac{p_m^{2r}+p_m^r}{p_m^{2r}+1}\leq\prod_{i=m+1}^{\infty}\left(1+\frac{1}{p_i^r}\right)\] for all $m\in\{1,2,3,4,6,9\}$. 
\end{theorem} 
\begin{proof} 
Let \[F(m,r)=\frac{p_m^{2r}+p_m^r}{p_m^{2r}+1}\prod_{i=1}^m\left(1+\frac{1}{p_i^r}\right)\] so that the inequality $\displaystyle{\frac{p_m^{2r}+p_m^r}{p_m^{2r}+1}\leq\prod_{i=m+1}^{\infty}\left(1+\frac{1}{p_i^r}\right)}$ is equivalent to $\displaystyle{F(m,r)\leq\frac{\zeta(r)}{\zeta(2r)}}$. Let $r\in(1,3]$. By Theorem~\ref{Thm2.1}, it suffices to show that if $\displaystyle{F(m,r)\leq\frac{\zeta(r)}{\zeta(2r)}}$ for all $m\in\{1,2,3,4,6,9\}$, then $\displaystyle{F(m,r)\leq\frac{\zeta(r)}{\zeta(2r)}}$ for all $m\in\mathbb{N}$. Therefore, assume that $r$ is such that $F(m,r)\leq\dfrac{\zeta(r)}{\zeta(2r)}$ for all $m\in\{1,2,3,4,6,9\}$. 

We will show that $F(m+1,r)>F(m,r)$ for all $m\in\mathbb N\setminus\{1,2,3,4,6,9\}$. This will show that $(F(m,r))_{m=10}^{\infty}$ is an increasing sequence. As 
$\displaystyle{\lim_{m\rightarrow\infty}F(m,r)=}$ $\displaystyle{\frac{\zeta(r)}{\zeta(2r)}}$, it will then follow that $\displaystyle{F(m,r)<\frac{\zeta(r)}{\zeta(2r)}}$ for all integers $m\geq 10$. Furthermore, we will see that
$F(5,r)<F(6,r)\leq\dfrac{\zeta(r)}{\zeta(2r)}$ and $F(7,r)<F(8,r)<\displaystyle{F(9,r)\leq\frac{\zeta(r)}{\zeta(2r)}}$, which will complete the proof. 

Let $m\in\mathbb{N}\backslash\{1,2,3,4,6,9\}$. By Lemma~\ref{Lem2.2},
$\dfrac{p_{m+1}}{p_m}<\sqrt[3]{2}\leq \sqrt[r]{2}$. This shows that $p_{m+1}^r<2p_m^r$, implying that $2p_m^{2r}>p_m^r p_{m+1}^r$. Therefore, 
\[2p_m^{2r}+2>p_m^r p_{m+1}^r+\frac{p_m^r}{p_{m+1}^r}-p_{m+1}^r-\frac{1}{p_{m+1}^r}=\frac{(p_m^r-1)(p_{m+1}^{2r}+1)}{p_{m+1}^r}.\] 
Multiplying each side of this inequality by $\displaystyle{\frac{p_{m+1}^r}{(p_{m+1}^{2r}+1)(p_m^{2r}+1)}}$ and adding $1$ to each side, we get 
\[1+\frac{2p_{m+1}^r}{p_{m+1}^{2r}+1}>1+\frac{p_m^r-1}{p_m^{2r}+1},\]
which we may write as 
\[\frac{(p_{m+1}^r+1)^2}{p_{m+1}^{2r}+1}>\frac{p_m^{2r}+p_m^r}{p_m^{2r}+1}.\] Finally, we get 
\[F(m+1,r)=\frac{p_{m+1}^{2r}+p_{m+1}^r}{p_{m+1}^{2r}+1}\prod_{i=1}^{m+1}\left(1+\frac{1}{p_i^r}\right)=\frac{(p_{m+1}^r+1)^2}{p_{m+1}^{2r}+1}\prod_{i=1}^m\left(1+\frac{1}{p_i^r}\right)\] 
\[>\frac{p_m^{2r}+p_m^r}{p_m^{2r}+1}\prod_{i=1}^m\left(1+\frac{1}{p_i^r}\right)=F(m,r).\qedhere\] 
\end{proof} 
Now, let \[V_m(r)=\log\left(\frac{p_m^{2r}+p_m^r}{p_m^{2r}+1}\right)-\sum_{i=m+1}^\infty\log\left(1+\frac{1}{p_i^r}\right).\] Equivalently, $\displaystyle{V_m(r)=\log(F(m,r))-\log\left(\frac{\zeta(r)}{\zeta(2r)}\right)}$, where $F$ is the function defined in the proof of Theorem~\ref{Thm2.2}. Observe that \[\frac{p_m^{2r}+p_m^r}{p_m^{2r}+1}\leq\prod_{i=m+1}^{\infty}\left(1+\frac{1}{p_i^r}\right)\] if and only if $V_m(r)\leq 0$. If we let $J_m(r)=\displaystyle{\sum_{i=m+1}^{m+6}\frac{1}{p_i^r+1}-\frac{p_m^{2r}-2p_m^r-1}{(p_m^r+1)(p_m^{2r}+1)}}$, then we have \[\frac{\partial}{\partial r}J_m(r)=\frac{p_m^r((p_m^r-1)^4-12p_m^{2r})\log p_m}{(p_m^{r}+1)^2(p_m^{2r}+1)^2}-\sum_{i=m+1}^{m+6}\frac{p_i^r\log p_i}{(p_i^r+1)^2}.\] It is not difficult to verify that $\displaystyle{\frac{p_m^r((p_m^r-1)^4-12p_m^{2r})\log p_m}{(p_m^{r}+1)^2(p_m^{2r}+1)^2}}\geq -1$ for all $r\in[1,2]$ and $m\in\{1,2,3,4,6,9\}$. Therefore, when $r\in[1,2]$ and $m\in\{1,2,3,4,6,9\}$, we have \[\frac{\partial}{\partial r}J_m(r)\geq -1-\sum_{i=m+1}^{m+6}\frac{p_i^r\log p_i}{(p_i^r+1)^2}\geq-1-\sum_{i=m+1}^{m+6}\frac{\log p_i}{p_i^r}> -7.\] Numerical calculations show that $\displaystyle{J_m(r)>\frac{1}{400}}$ for all $m\in\{1,2,3,4,6,9\}$ and \[r\in\left\{1+\frac{n}{2800}\colon n\in\{0,1,2,\ldots,2800\}\right\}.\] Because each function $J_m$ is continuous in $r$ for $r\in[1,2]$, we see that \[J_m(r)>\frac{1}{400}-7\left(\frac{1}{2800}\right)=0\] for all $r\in[1,2]$ and $m\in\{1,2,3,4,6,9\}$.   

We introduced the functions $J_m$ so that we could write  
\[\frac{\partial}{\partial r}V_m(r)=\sum_{i=m+1}^{\infty}\frac{\log p_i}{p_i^r+1}-\frac{(p_m^{2r}-2p_m^r-1)\log p_m}{(p_m^r+1)(p_m^{2r}+1)}>(\log p_m)J_m(r)>0\] for all $m\in\{1,2,3,4,6,9\}$ and $r\in[1,2]$. 
A quick numerical calculation shows that $V_2(1.5)<0<V_2(2)$, so the function $V_2$ has exactly one root, which we will call $\eta^*$, in the interval $(1,2]$. Further calculations show that $V_m(2)<0$ for all $m\in\{1,3,4,6,9\}$. Hence, $V_m(r)\leq 0$ for all $m\in\{1,2,3,4,6,9\}$ and $r\in(1,\eta^*]$. By Theorem~\ref{Thm2.2}, this means that if $r\in(1,2]$, then $\overline{\sigma_{-r}^*(\mathbb N)}\left[1,\zeta(r)/\zeta(2r)\right]$ if and only if $r\leq\eta^*$. 

Next, note that \[\frac{\partial}{\partial r}V_2(r)=\sum_{i=3}^{\infty}\frac{\log p_i}{p_i^r+1}-\frac{(3^{2r}-2\cdot 3^r-1)\log 3}{(3^{2r}+1)(3^r+1)}>-\frac{(3^{2r}-2\cdot 3^r-1)\log 3}{(3^{2r}+1)(3^r+1)}\] 
\[>-\frac{(3^{2r}+1)\log 3}{(3^{2r}+1)(3^r+1)}\geq -\frac{\log 3}{3^2+1}>-1.1\]
for all $r\in[2,3]$. Let $\displaystyle{A=\left\{2+\frac{n}{400}\colon n\in\{0,1,2,\ldots,400\}\right\}}$. With a computer program, one may verify that $V_2(r)>0.003$ for all $r\in A$. Because $V_2$ is continuous, this shows that $V_2(r)>0.003-1.1\displaystyle{\left(\frac{1}{400}\right)}>0$ for all $r\in [2,3]$. Consequently, $\overline{\sigma_{-r}^*(\mathbb N)}\neq\displaystyle{\left[1,\zeta(r)/\zeta(2r)\right)}$ if $r\in[2,3]$. 

We are now in a position to prove Theorem~\ref{Thm2.3}. Note that the equation defining $\eta^*$ in the statement of this theorem is simply a rearrangement of the equation $V_2(\eta^*)=0$. Therefore, we have shown that the theorem is true for $r\in(1,3]$. In order to prove the theorem for $r>3$, it suffices (by Theorem~\ref{Thm2.2}) to show that $\displaystyle{F(1,r)>\frac{\zeta(r)}{\zeta(2r)}}$ for all $r>3$. If $r>3$, then  
\[F(1,r)=\frac{(2^r+1)^2}{2^{2r}+1}=\frac{2^{2r}+2^{r+1}+1}{2^{2r}+1}>\frac{2^{2r}+2^r+\frac{2^{r+1}}{r-1}}{2^{2r}+1}=\frac{1+\frac{1}{2^r}+\frac{1}{(r-1)2^{r-1}}}{1+\frac{1}{2^{2r}}}\] 
\[>\frac{1+\frac{1}{2^r}+{\frac{1}{(r-1)2^{r-1}}}}{\zeta(2r)}=\frac{1+\frac{1}{2^r}+\int_2^{\infty}x^{-r}dx}{\zeta(2r)}>\frac{\zeta(r)}{\zeta(2r)}.\]   
 
\section{Future Directions} 
Let $\mathcal N^*(t)$ denote the number of connected components of $\overline{\sigma_t^*(\mathbb N)}$. It would be interesting to obtain analogues of Zubrilina's results \cite{Nina} by finding asymptotic estimates for $\mathcal N^*(-r)$ as $r\to\infty$. Let \[E_m^*=\{t\in\mathbb R\colon\mathcal N^*(t)=m\}.\] Theorem~\ref{Thm2.3} tells us that $E_1^*=[-\eta^*,0)$. The sets $E_m^*$ are the natural unitary analogues of the sets $E_m$ defined in \cite[Section 4]{Defant5}. Continuing the analogy, we say a positive integer $m$ is a \emph{unitary Zubrilina number} if $E_m^*=\emptyset$ (the name comes from Zubrilina's result that $E_4=\emptyset$). We do not have any specific examples of unitary Zubrilina numbers, but we still make the following conjectures. 

\begin{conjecture}\label{Conj1}
There are infinitely many unitary Zubrilina numbers.
\end{conjecture}
  
\begin{conjecture}\label{Conj2}
For $r>1$, $\mathcal N^*(-r)$ is monotonically increasing as a function of $r$. 
\end{conjecture}

Note that Conjecture \ref{Conj2} implies that the sets $E_m^*$ are intervals. 

\section{Acknowledgements} 
The author thanks the referee for carefully reading the manuscript and providing very helpful suggestions.

\end{document}